\documentclass[11pt,reqno,british,a4paper,twoside]{amsart}
\usepackage[sc]{mathpazo}
\usepackage{amsmath,mathtools,amssymb,url,paralist,xspace,babel,mathabx}
\usepackage[capitalise,noabbrev]{cleveref}
\usepackage{booktabs,microtype}
\usepackage{graphicx}
\usepackage{caption,subcaption}
\mathtoolsset{mathic=true}
\setdefaultenum{\upshape (i)}{}{}{}
\usepackage[margin=1.2in]{geometry}
\usepackage{comment}

\theoremstyle{plain}
\newtheorem{theorem}{Theorem}[section]
\newtheorem{proposition}[theorem]{Proposition}

\theoremstyle{definition}
\newtheorem{definition}[theorem]{Definition}
\newtheorem{remark}[theorem]{Remark}
\newtheorem{example}[theorem]{Example}

\numberwithin{equation}{section}
\numberwithin{table}{section}

\newcommand{\Lie}[1]{\operatorname{\rm{#1}}}
\newcommand{\lie}[1]{\operatorname{\mathfrak{#1}}}
\newcommand{\G}{\Lie{G}}
\newcommand{\LK}{\Lie{K}}
\newcommand{\GL}{\Lie{GL}}
\newcommand{\SO}{\Lie{SO}}
\newcommand{\Spin}{\Lie{Spin}}
\newcommand{\Sp}{\Lie{Sp}}
\newcommand{\SU}{\Lie{SU}}

\newcommand{\g}{\lie{g}}
\newcommand{\lk}{\lie{k}}
\newcommand{\so}{\lie{so}}
\newcommand{\su}{\lie{su}}
\newcommand{\lsp}{\lie{sp}}

\newcommand{\Hodge}{{*}}
\newcommand{\hook}{{\lrcorner\,}}

\newcommand{\bC}{{\mathbb C}}
\newcommand{\bH}{{\mathbb H}}
\newcommand{\bR}{{\mathbb R}}

\newcommand{\sS}{{\mathcal S}}

\begin{document}

\title[Instantons on flat space]{Instantons on flat space: Explicit constructions}

\author{Jason D.~Lotay} \address[J.\,D.~Lotay]{Mathematical Institute\\ University of Oxford\\ Woodstock Road\\ Oxford OX2 6GG\\ United Kingdom.}  \email{jason.lotay@maths.ox.ac.uk}

\author{Thomas Bruun Madsen}
\address[T.\,B.~Madsen]{School of Computing and Engineering\\ University of West London\\
  St Mary's Road, Ealing\\ London W5 5RF\\ United Kingdom.}
\email{thomas.madsen@uwl.ac.uk}

\begin{abstract}
In this note, we revisit some well-known examples of instantons on flat space that were originally discovered in
the physics literature. In particular, we explain how the basic instanton on \( \bR^4 \), with its flat hyperk\"ahler structure, has natural generalisations to \( \bR^7 \) and \( \bR^8 \) viewed as
flat \( \G_2 \)- and \( \Spin(7) \)-manifolds, respectively. We also provide the details of an arguably less known construction of ASD instantons on \( \bH^n \), in the sense of quaternionic geometry.
\end{abstract}

\maketitle

\section{Introduction}
\label{sec:intro}
The idea of generalising the concept of (anti-)self-duality to higher dimensions has been around for about four decades, gaining increased attention
with the advances in special and exceptional holonomy (see, e.g.,  \cite{Karigiannis-al:G2Survey} for a recent collection of surveys) and the potential link between higher-dimensional gauge theory and enumerative invariants, initiated in the seminal papers \cite{DonaldsonThomas:higherGauge,DonaldsonSegal:higherGauge}. 
Instantons in higher dimensions, which are connections whose curvature satisfy these generalised anti-self-duality conditions, have been explored both on non-compact and compact spaces, with most of the current interest centering around \( \G_2 \)-instantons, e.g., \cite{LotayOliveira:G2inst,Nordstrom-al:G2inst,SaEarpWalpuski:TCS,Walpuski:Kummer,Walpuski:G2Fueter}. 

The purpose of this note is to revisit
some of the early constructions of instantons  that appeared in the physics literature \cite{BPST:YangMills,Fairlie-N:instanton,Ivanova-P:instanton,Nicolai-G:instanton}, rephrasing explicit constructions in a way that highlights similarities and contextualising them in view of more recent developments. In particular, we shall explain that the basic (or BPST) instanton on \( \bR^4 \), with its flat hyperk\"ahler structure, has natural generalisations to \( \bR^7 \) and \( \bR^8 \) viewed as
flat \( \G_2 \)- and \( \Spin(7) \)-manifolds, respectively.  This generalization is achieved via an ansatz based on the groups associated with the ambient flat structure. Surprisingly, however, this natural ansatz that we present does not produce instantons for \(  \bR^n \) equipped with the remaining Ricci-flat geometries from Berger's list \cite{Berger:Holonomy}. 

\subsection{\(\G \)-instantons} To set the scene, consider having a principal \( \LK \)-bundle
\( P\to M \)
over an oriented Riemannian \( n \)-manifold
\( M \).
Given a connection form \( \omega\in \Omega^1(P;\lk) \)
the associated curvature will be \( \Omega\in\Omega^2(M;\lk_P) \),
where \( \lk \)
is the Lie algebra of \(\LK\).
When \( M \)
comes equipped with a \( \G \)-structure, for some
\( \G\subset\SO(n) \), then since we can identify the 2-forms on \(M\) at each point with the Lie algebra of \( \SO(n) \) we can decompose the bundle of \( 2 \)-forms as
\begin{equation*}
  \Lambda^2T^*M\cong \so(n)\cong \g\oplus \g^\perp,
\end{equation*}
where the fibres of \( \g \)
are given by the Lie algebra of \( \G \).
This splitting gives us a way of distinguishing connections that are
particularly adapted to the geometry (cf.~\cite{Carrion:Instanton}).

\begin{definition}
  \label{def:instanton-def}
  A connection \( \omega \)
  on \( P \)
  is called a \( \G \)-\emph{instanton}
  if the \( 2 \)-form
  part of its curvature \( \Omega \)
  takes values in the subbundle \( \g\subset\Lambda^2T^*M \).
\end{definition}

A natural setting where we have a distinguished \( \G \)-structure
on \( M \)
is when the metric on \( M \)
has special holonomy \( \G \)
(and thus the \( \G \)-structure
is torsion-free). Of course, holonomy reduction is trivially obtained on flat space \( M=\bR^n \) which is our focus in this note.  Taking the group \( \G \) in question to be \(\SU(2)\), \(\G_2\) and \(\Spin(7)\) in dimensions \( 4 \), \( 7 \) and \( 8 \) leads to the notions of (anti-)self-dual, \(\G_2\)- and \(\Spin(7)\)-instantons, respectively.

\subsection{Results} 
The main observation of this note is that there is an ansatz which gives a unified perspective on the construction of the basic \(\G\)-instantons on flat spaces. 

Let \( P_{\G}\cong\G\times\bR^n \) be the reduction of the principal frame bundle of \( \bR^n \)
corresponding to a choice of \( \G \)-structure on \(\bR^n\).  We  construct a family of  connections on \(P_{\G}\) as follows. 

Away from the origin, we have the usual identification
\( \bR^n\setminus\{0\}\cong\bR_+\times \sS^{n-1}\),
with 
radial coordinate \( r\) of \( \bR_+ \) and we write \( \partial_r=\partial/\partial r\).
Let \( e_j \), for
\( j=1,\ldots,\dim\G \),
define a Killing orthonormal basis of the Lie algebra
\( \g \)
of \( \G \)
inside \( \so(n) \)
and denote by \( \beta_j \)
the elements defining a corresponding basis for \( \g \) inside \( \Lambda^2(\bR^n)^*\cong\so(n) \).
Then define \( 1 \)-forms
\( \alpha_j \)
on \( \sS^{n-1} \)
via 
\begin{equation*}
\alpha_j=\partial/\partial_r\hook\beta_j|_{\sS^{n-1}} .
\end{equation*}
 We let \( \omega\in\Omega^1(P_{\G};\lie g) \) be of the form
\begin{equation}
  \label{eq:conn-ansatz}
  \omega=A(r)=ra(r)\sum_{j=1}^{\dim\G}\alpha_j\otimes e_j
\end{equation} 
for some function \(a:\bR_+\to\bR\), where we may view \( \omega \) as a family of connections \( A(r) \) on \( \G\times\mathcal{S}^{n-1} \).  We have added the factor of \( r \) in \eqref{eq:conn-ansatz} for convenience.  

We then show the following.

\begin{theorem}\label{thm:basic}
For \(\G=\SU(2),\G_2,\Spin(7)\) for \( n=4,7,8 \) respectively, the ansatz \eqref{eq:conn-ansatz} yields a \( \G \)-instanton on \( \bR^n \) with structure group \( \G \) for
\begin{equation*}
 a(r)=c_n\frac{r}{r^2+C},
\end{equation*}
where \( C>0 \) is a constant and \( c_n \) is a fixed constant only depending on \( n \).  These are the basic \( \G \)-instantons on \( \bR^n \) for \( n=4,7,8\).
\end{theorem}

As observed in \cref{prop:YMrInst4d}, in the \( 4 \)-dimensional case, these instanton solutions are the only globally defined Yang-Mills connections satisfying our rotationally symmetric ansatz  \eqref{eq:conn-ansatz}.

\begin{remark}
\label{rem:basicansatz}
Given the result above, it is  natural to ask what happens for   the other Ricci-flat geometries, associated with groups \( \G=\SU(n) \), \( n>2 \), on \( \bR^{2n}=\bC^n \) and \( \Sp(n) \), \( n>1 \), on \( \bR^{4n}=\bH^n \).   
However, direct computations show that the ansatz \eqref{eq:conn-ansatz} does not produce \( \G \)-instantons in these situations.  This is perhaps another manifestation of the ``exceptional'' nature of the holonomy groups \( \G_2 \) and \( \Spin(7)\), as well as the special case of 4 dimensions where the group is \( \SU(2)=\Sp(1) \).
\end{remark}

We also discuss another construction of instantons on flat space which is related to Nahm's equations and  quaternionic structures. For any solution of Nahm's equations, we obtain families of \( \Sp(n) \)-instantons on \( \bH^n \), as explained by \cref{prop:NahmHn}.

\begin{remark}
Throughout the article we will use \( (x_1,\ldots,x_n) \) for coordinates on \( \bR^n \) and let \( dx^{ij\ldots k} \) denote \( d x^i\wedge dx^j\wedge\ldots\wedge dx^k \). We also let \( E_{k\ell} \) be 
the elementary \( n\times n \) matrix with \( 1 \)
in the \((k,\ell) \) entry 
and zero for all other entries.
\end{remark}

\noindent{\bf Acknowledgements.}  JDL is partially supported by the Simons Collaboration on Special Holonomy in Geometry, Analysis, and Physics (\( \#724071 \) Jason Lotay).

\section{Basic instantons}

We will now show explicitly how the ansatz \eqref{eq:conn-ansatz} recovers the basic instanton on \( \bR^4 \) (with the flat metric) and leads to generalisations 
on \( \bR^7 \) (with a flat \( \G_2 \)-structure) and \( \bR^8 \) (with a flat \( \Spin(7) \)-structure).

\subsection{On \( \bR^4 \)}
We first consider the construction of the basic instanton
on the trivial bundle \( \SU(2)\times \bR^{4} \)
and so need the standard \( \SU(2) \)-structure on \( \bR^{4}\cong\bC^2 \).
This structure can be specified in terms of a non-degenerate \( 2 \)-form
\( \sigma \) and a \( (2,0) \)-form \( \Psi \) on \(\bR^4\) given by:
\begin{equation*}
  \begin{gathered}
    \sigma=dx^{12}+dx^{34}
    ,
    \quad
    \Psi=\psi+i\hat\psi    =(dx^1+idx^2)\wedge(dx^3+idx^4).
  \end{gathered}
\end{equation*}
With a view to later constructions, we can see this in terms of structures on
\( \bR^4\setminus\{0\}\cong\bR_+\times \sS^3 \), we have
\begin{equation}\label{eq:s3.decomp}
  \begin{gathered}
    \sigma=dr\wedge r\eta+\frac{r^2}{2} d\eta,\quad \Psi=(dr+ir\alpha)\wedge
    r\Phi,
  \end{gathered}
\end{equation}
for a (contact) \(1\)-form \(\eta\)  
and a \((1,0)\)-form \(\Phi\) on \(\sS^3\).
For later use, it is convenient to introduce the notation
\( \Phi=\phi+i\hat\phi \)
for the decomposition of the \( (1,0) \)-form
\( \Phi \) into real and imaginary parts.

The condition for a \( 2 \)-form
\( \beta \)
to be in \( \su(2)\subset\Lambda^2(\bR^4)^* \) can be expressed as
\begin{equation}
  \label{eq:sun-cond}
  \Psi\wedge\beta=0\quad\text{and}\quad \sigma\wedge\beta=0.
\end{equation}

\begin{remark}\label{rmk:sd} It is elementary to check that the condition \eqref{eq:sun-cond} on \( \beta \) is equivalent to demanding that \( \beta\) is anti-self-dual with respect to the standard orientation on \(\bR^4 \).  The first equality in \eqref{eq:sun-cond} says that \( \beta \)
is of type \( (1,1) \)
for the standard complex structure on \( \bR^4\cong\bC^2 \), and we note that, as \( \beta \)
is real, it suffices to check that \( \beta\wedge\psi=0 \).
The second equality in \eqref{eq:sun-cond} implies that \( \beta \)
is ``trace-free'' and so orthogonal to \( \sigma \)
inside \( \Lambda^2(\bR^4)^* \).
\end{remark}

In order to proceed, let us fix a basis of \( \su(2)\leqslant\so(4) \). 
Recalling the elementary matrices \( E_{k\ell} \), an orthonormal basis, up to
scale, of \( \su(2) \) is given by the elements:
\begin{equation*}
  \begin{gathered}
  e_1=\frac{1}{2}(-E_{12}+E_{21}+E_{34}-E_{43}),\quad
  e_2=\frac{1}{2}(-E_{13}+E_{31}+E_{42}-E_{24}),\\
  e_3=\frac{1}{2}(-E_{14}+E_{41}+E_{23}-E_{32}).
  \end{gathered}
\end{equation*}
This basis satisfies the usual commutation relations \( [e_i,e_j]=-2e_k
\) for cyclic permutations \( (i,j,k) \) of \( (1,2,3) \).
Then \( \su(2)\subset\Lambda^2(\bR^4)^* \) is spanned by the corresponding \( 2 \)-forms
\begin{equation*}
  \begin{gathered}
  \beta_1=dx^{12}-dx^{34},\quad \beta_2=dx^{13}-dx^{42},\quad\beta_3=dx^{14}-dx^{23}.
  \end{gathered}
\end{equation*}

Now let \( \omega \)
be a connection given by~\eqref{eq:conn-ansatz} with curvature \( \Omega \).
It follows by
\cref{def:instanton-def} and \eqref{eq:sun-cond} that the
\( \SU(2) \)-instanton
condition is given by 
\begin{equation*}
 \Omega \wedge \sigma=0 \quad\text{
and} \quad \Omega\wedge \psi=0 .
\end{equation*}
Equivalently, viewing \( \omega \) as family of connnections \( A(r) \)
over \( \sS^3 \) and recalling \eqref{eq:s3.decomp}, we can express the instanton condition as
\begin{equation}\label{eq:Sun-instanton-eq2}
\begin{gathered}
  rA'\wedge d\eta=-2F_A\wedge\eta,\quad F_A\wedge d\eta=0,\\
    rA'\wedge\alpha\wedge \hat\phi=F_A\wedge\phi,\quad
    F_A\wedge\hat\phi\wedge\alpha=0,
  \end{gathered}
\end{equation}
where \( F_A \) is the curvature of \( A \).

Given the symmetry of the problem, for computational simplicity we can consider the point
\( (1,0,0,0)\) in \(\sS^3 \)
so that we may identify the unit radial vector field 
\( \partial_r \)
with \( \partial_1=\partial/\partial x_1 \).  We then have that
\( \alpha_j=dx^{j+1} \) and
\begin{equation*}
  \begin{gathered}
    \beta_1|_{\sS^3}=-dx^{34}
    ,\quad
    \beta_2|_{\sS^3}=-dx^{42}
    ,\quad
    \beta_3|_{\sS^3}= -dx^{23}
    .
  \end{gathered}
\end{equation*}
Now at \( (1,0,0,0)\in \sS^3 \) the \( 2 \)-form \( F_A \) is given by:
\begin{equation*}
F_A=2(ra(r)(1+ra(r))\sum_{j=1}^3\beta_j|_{\mathcal{S}^3}\otimes e_j.
\end{equation*}

It is now easy to see that the  instanton condition as expressed in \eqref{eq:Sun-instanton-eq2} reduces
to a single ODE:
\begin{equation}
  \label{eq:SU2-inst}
  \frac{da}{dr}=\frac{a}{r}(1+2ar).
\end{equation}
Solving \eqref{eq:SU2-inst}, we recover the basic instanton on
\( \bR^4 \). In summary:

\begin{proposition}
  \label{prop:basicinstanton-R4}
  The basic instanton on \( \bR^4 \) arises from the ansatz \eqref{eq:conn-ansatz} with \(\G=\SU(2)\) and
  \begin{equation*}
    a(r)=-\frac{r}{r^2+C},
  \end{equation*}
  where \( C \) is a positive constant.
  \qed
\end{proposition}

\begin{remark}
It follows from the above computations that 
\( \int_{\bR^{4}}\|\Omega\|^2{\rm{vol}}_{\bR^{4}} \)
is finite. Note that the \(\SU(2)\)-instanton condition is conformally invariant since it is the anti-self-duality condition by Remark \ref{rmk:sd}.  Since the \( L^2 \)-norm of \( \Omega \) is also conformally invariant, it follows that stereographic projection will take the basic instanton on \( \bR^4 \) and give an instanton on \( \sS^{4} \)
with a removable point singularity. 
\end{remark}

It seems natural to also ask, more generally, what are the possible solution to 
the Yang--Mills equation, arising from our connection ansatz \eqref{eq:conn-ansatz} in this setting; that is, connections which are critical points for the Yang--Mills functional.   
Elementary computations show that the Yang--Mills equations are given by the second
order non-linear ODE:
\begin{equation}
\label{eq:YM4d}
r^2\frac{d^2a}{dr^2}+3r\frac{da}{dr}-4ra^2(3+2ra)-3a=0.
\end{equation}
To better understand this second order ODE, it is useful make the change of variables: 
\begin{equation*}
u(r)=\frac{a(r)}{r}.
\end{equation*}
Note that, in terms of \( u \), the instanton condition \eqref{eq:SU2-inst} reads:
\begin{equation*}
\frac{du}{dr}=2ru^2.
\end{equation*}
Motivated by this, let: 
\begin{equation*}
v(r)=r^5\left(\frac{du}{dr}-2ru^2\right).
\end{equation*}
Using these variables, \( (u,v) \), the Yang-Mills condition \eqref{eq:YM4d} becomes equivalent to the following system of first order ODEs:
\begin{equation}
\label{eq:YM1stODEsys}
\frac{du}{dr}=2ru^2+\frac{v}{r^5}\quad\textrm{and}\quad
\frac{dv}{dr}=-4ruv.
\end{equation}

Now, it is not difficult to see that the instanton solutions of \cref{prop:basicinstanton-R4} exhaust all globally defined Yang-Mills solutions.

\begin{proposition}
\label{prop:YMrInst4d}
The only globally defined Yang-Mills solutions of \eqref{eq:YM4d} are the instanton solutions of \cref{prop:basicinstanton-R4}.
\end{proposition}

\begin{proof}
The statement follows if we can show that a globally defined solution \( (u,v) \) of \eqref{eq:YM1stODEsys} necessarily has \( v\equiv0 \). 
In this case, \( a(r)= ru(r) \) solves the instanton equation \eqref{eq:SU2-inst}. By uniqueness of soutions to (regular) initial value problems, \( v\equiv0 \) holds if we can show that \( v(r_0)=0 \) for some \( r_0>0 \).

To start with, we observe that regularity of the system \eqref{eq:YM1stODEsys}, at zero, forces \( v(0)=0 \). Next choose \( \varepsilon>0 \) so that either \( u(r)\geqslant0 \), or \( u(r)\leqslant0 \), for all \( [0,\varepsilon] \). 

In the case where \( u(r)\geqslant0 \), we see that for \( r\in[0,\varepsilon] \), \( v'(r) \) will either vanish or have the opposite sign of \( v(r) \).
Given that \( v(0)=0 \), this is only possible if \( v \) vanishes identically. So, in  this case, \( v(r_0)=0\) for some \( r_0>0 \), as required. 
 
Similarly, if  \( u(r)\leqslant0 \) on \( [0,\varepsilon] \), choose \( \delta\in\{-1,1\} \) so that
\( \delta v(r)\geqslant0 \) for small enough \( r \). Let \( C\geqslant 0 \) be a constant so that 
\( 0\leqslant -4ru\leqslant C \) on \( [0,\varepsilon] \). Then, for small \( r \), we have that:
\begin{equation*}
\delta v(r)=\delta\left( v(r)- v(0)\right)=\delta \int_0^r\frac{dv}{ds}ds=\int_0^r(-4su )\delta vds\leqslant C\Big\lvert\int_0^r\delta vds\Big\rvert.
\end{equation*}
Hence, by Gr{\"o}nwall's inequality \(\delta v\equiv0 \), so \( v\equiv0 \), for small enough values of \( r \). So again, we have that
\( v(r_0)=0 \) for some \( r_0>0 \).

In conclusion, we have shown that for any globally defined solution \( (u,v) \) of \eqref{eq:YM1stODEsys}, \( a(r)=ru(r) \) is a solution of the instanton condition \eqref{eq:SU2-inst}, as required.
\end{proof}

\subsection{On \( \bR^7 \)}

We now turn to the construction of what we shall call the basic instanton on the
trivial bundle \( \G_2\times\bR^7 \).
This requires a geometric description of the compact Lie group
\( \G_2 \) which may be given in terms of the \( 3 \)-form \( \varphi \) on \(\bR^7\) given by
\begin{equation}
  \label{eq:G2-3form}
  \varphi=dx^{123}+dx^1\wedge(dx^{45}+dx^{67})+dx^2\wedge(dx^{46}+dx^{75})-dx^3\wedge(dx^{47}+dx^{56}),
\end{equation}
as \( {\rm{Stab}}_{\GL(7,\bR)}(\varphi)=\G_2\subset\SO(7) \).
Since \( \varphi \)
determines the standard metric and orientation on \( \bR^7 \),
we can also consider its dual \( 4 \)-form:
\begin{equation}
  \label{eq:G2-4form}
  \Hodge\varphi=dx^{4567}+dx^{23}\wedge(dx^{45}+dx^{67})+dx^{31}\wedge(dx^{46}+dx^{75})-dx^{12}\wedge(dx^{47}+dx^{56}).
\end{equation}
Rephrasing in terms of the standard nearly K\"ahler structure
\( (\sigma,\Psi=\psi+i\hat{\psi}) \) on the unit sphere \( \sS^6\subset \bR^7 \), where \( \sigma \) is a \( 2 \)-form and \( \Psi \) is a \( (3,0) \)-form, we have
\begin{equation}\label{eq:s6}
  \varphi=dr\wedge r^2\sigma+r^3\psi,\quad \Hodge\varphi=r^3\hat\psi\wedge dr+\tfrac12 r^4\sigma^2.
\end{equation}
on \( \bR_+\times \sS^6 \).

It is well-known that the condition for a \( 2 \)-form
\( \beta \)
to be in \( \g_2\subset\Lambda^2(\bR^7)^* \)
can be expressed as either of the equivalent conditions
\begin{equation}\label{eq:g2}
  \Hodge(\varphi\wedge\beta)=-\beta
  \quad\textrm{or}\quad
  \Hodge\varphi\wedge\beta=0.
\end{equation}
The first condition manifestly has echoes of the anti-self-duality condition in 4 dimensions.

For an \( \G_2 \)-instanton of the form \eqref{eq:conn-ansatz}, as well as substituting \( \beta=\Omega \) in \eqref{eq:g2}, the instanton
condition can also be rephrased as:
\begin{equation}\label{eq:G2-evol}
  A'=\Hodge_r(F_A\wedge r^3\psi),
\end{equation}
subject to the initial constraint 
\begin{equation}\label{eq:G2-constraint}
F_A(r_0)\wedge\sigma(r_0)^2=0,
\end{equation}
see, e.g., \cite{Lotay-M:instantons}.  Here we are using the decomposition \eqref{eq:s6} and \( \Hodge_r \) denotes the Hodge star on the 6-sphere of radius \( r \).

As in the previous case, to make our construction explicit, we need
to establish an identification of \( \g_2\leqslant\so(7)\).
Using the elementary matrices \( E_{k\ell} \), a basis of \( \g_2 \)
is given as follows:
\begin{align*}
  e_1&=
  \tfrac14(E_{23}-E_{32}-E_{45}+E_{54}),
  & e_2&=
  -\tfrac1{4\sqrt3}(E_{23}-E_{32}+E_{45}-E_{54}
  -2E_{67}+2E_{76}),\\
  e_3&=\tfrac14(E_{13}-E_{31}+E_{46}-E_{64})
  ,& e_4&= 
  -\tfrac1{4\sqrt3}(E_{13}-E_{31}-E_{46}+E_{64}
  -2E_{57}+2E_{75})
  ,\\
  e_5&=
  -\tfrac14(E_{12}-E_{21}+E_{47}-E_{74})
  ,& e_6&=
  \tfrac1{4\sqrt3}(E_{12}-E_{21}-E_{47}+E_{74}
  +2E_{56}-2E_{65}),\\
  e_7&= 
  \tfrac14(E_{15}-E_{51}-E_{26}+E_{62})
  ,&
  e_8&=
  \tfrac1{4\sqrt3}(E_{15}-E_{51}+E_{26}-E_{62}
  +2E_{37}-2E_{73})
  ,
\\
  e_9&=
  -\tfrac14(E_{14}-E_{41}-E_{27}+E_{72}),&
  e_{10}&=
  -\tfrac1{4\sqrt3}(E_{14}-E_{41}+E_{27}-E_{72}
  -2E_{36}+2E_{63}),\\
  e_{11}&=
  \tfrac14(E_{17}-E_{71}+E_{24}-E_{42})
  ,&
  e_{12}&=
  \tfrac1{4\sqrt3}(E_{17}-E_{71}-E_{24}+E_{42}
  -2E_{35}+2E_{53})
  ,\\
  e_{13}&=
  -\tfrac14(E_{16}-E_{61}+E_{25}-E_{52}),&
  e_{14}&=
  -\tfrac1{4\sqrt3}(E_{16}-E_{61}-E_{25}+E_{52}
  +2E_{34}-2E_{43}).
\end{align*}
Correspondingly \( \g_2\subset\Lambda^2(\bR^7)^* \) has basis:
\begin{align*}
    \beta_1&=-\tfrac14(dx^{23}-dx^{45}),&
    \beta_2&=\tfrac1{4\sqrt3}(dx^{23}+dx^{45}-2dx^{67}),\\
    \beta_3&=-\tfrac14(dx^{13}+dx^{46}),&
    \beta_4&=\tfrac1{4\sqrt3}(dx^{13}-dx^{46}-2dx^{57}),\\
    \beta_5&=\tfrac14(dx^{12}+dx^{47}),&
    \beta_6&=-\tfrac1{4\sqrt3}(dx^{12}-dx^{47}+2dx^{56}),\displaybreak[0]\\
    \beta_7&=-\tfrac14(dx^{15}-dx^{26}),&
    \beta_8&=-\tfrac1{4\sqrt3}(dx^{15}+dx^{26}+2dx^{37}),\displaybreak[0]\\
    \beta_9&=\tfrac14(dx^{14}-dx^{27}),&
    \beta_{10}&=\tfrac1{4\sqrt3}(dx^{14}+dx^{27}-2dx^{36}),\displaybreak[0]\\
    \beta_{11}&=-\tfrac14(dx^{17}+dx^{24}),&
    \beta_{12}&=-\tfrac1{4\sqrt3}(dx^{17}-dx^{24}-2dx^{35}),\displaybreak[0]\\
  \beta_{13}&=\tfrac14(dx^{16}+dx^{25}),&
  \beta_{14}&=\tfrac1{4\sqrt3}(dx^{16}-dx^{25}+2dx^{34}).
\end{align*}

For simplification, we proceed by considering the point \( x\in\mathcal{S}^6 \) with \( x_1=1 \) and \( x_j=0 \) for all other \( j \). Then at \( x \) we have:
\begin{equation*}
\begin{gathered}
  \alpha_1=0,\quad \alpha_2=0,\quad
  \alpha_3=-\tfrac14dx^3,\quad
  \alpha_4=\tfrac1{4\sqrt3}dx^3,\quad
  \alpha_5=\tfrac14dx^2,\\
  \alpha_6=-\tfrac1{4\sqrt3}dx^2,\quad
  \alpha_7=-\tfrac14dx^5,\quad
  \alpha_8=-\tfrac1{4\sqrt3}dx^5,\quad
  \alpha_9=\tfrac14dx^4,\quad
  \alpha_{10}=\tfrac1{4\sqrt3}dx^4,\\
  \alpha_{11}=-\tfrac14dx^7,\quad
  \alpha_{12}=-\tfrac1{4\sqrt3}dx^7,\quad
  \alpha_{13}=\tfrac14dx^6,\quad
  \alpha_{14}=\tfrac1{4\sqrt3}dx^6.
  \end{gathered}
\end{equation*}
Similarly, at \( x \), we have:
\begin{equation*} \begin{aligned}
    \beta_1|_{\sS^6}&=-\tfrac14(dx^{23}-dx^{45}),&
    \beta_2|_{\sS^6}&=\tfrac1{4\sqrt3}(dx^{23}+dx^{45}-2dx^{67}),\\
    \beta_3|_{\sS^6}&=-\tfrac14dx^{46},&
    \beta_4|_{\sS^6}&=\tfrac1{4\sqrt3}(-dx^{46}-2dx^{57}),\\
    \beta_5|_{\sS^6}&=\tfrac14dx^{47},&
    \beta_6|_{\sS^6}&=\tfrac1{4\sqrt3}(dx^{47}-2dx^{56}),\\
    \beta_7|_{\sS^6}&=\tfrac14dx^{26},&
    \beta_8|_{\sS^6}&=-\tfrac1{4\sqrt3}(dx^{26}+2dx^{37}),\\
    \beta_9|_{\sS^6}&=-\tfrac14dx^{27},&
    \beta_{10}|_{\sS^6}&=\tfrac1{4\sqrt3}(dx^{27}-2dx^{36}),\\
    \beta_{11}|_{\sS^6}&=-\tfrac14dx^{24},&
    \beta_{12}|_{\sS^6}&=\tfrac1{4\sqrt3}(dx^{24}+2dx^{35}),\\
    \beta_{13}|_{\sS^6}&=\tfrac14dx^{25},&
    \beta_{14}|_{\sS^6}&=\tfrac1{4\sqrt3}(-dx^{25}+2dx^{34}).  \end{aligned} \end{equation*}
    
Using these explicit formulae, we consider the evolution equation \eqref{eq:G2-evol} subject to the initial constraint \eqref{eq:G2-constraint} for the ansatz \eqref{eq:conn-ansatz}.  
If we look at the \( e_5 \) component of \( A' \), \( \Hodge_r(dA\wedge r^4\psi) \) and \( \Hodge_r(\tfrac12[A,A]\wedge r^4\psi) \) at \( x \)  we see that they are
\begin{equation*}
  \tfrac14\frac{d(ar)}{d r}dx^2,\quad \tfrac142adx^2,\quad \tfrac14\tfrac16a^2rdx^2,
\end{equation*}
respectively. This suggests that the evolution equation \eqref{eq:G2-evol}  is equivalent to the ODE:
\begin{equation}\label{eq:g2-ode}
  \frac{da}{d r}=\frac{a}{r}(1+\tfrac16ar).
\end{equation}
Performing similar computations for the remaining \( 13 \) components allow us to verify that this is indeed the case.  
It is also straightforward to check that the constraint 
\eqref{eq:G2-constraint} holds automatically at \( x \), the \( \G_2 \)-instanton condition for the ansatz \eqref{eq:conn-ansatz} is equivalent to solving \eqref{eq:g2-ode}.

In conclusion, we have shown the following.
\begin{proposition}
  \label{prop:instanton-G2}
  The ansatz \eqref{eq:conn-ansatz} on \( \bR^7 \) with \(\G=\G_2\) yields a non-trivial \( \G_2 \)-instanton provided
  \begin{equation*}
    a(r)=-12\frac{r}{r^2+C},
  \end{equation*}
  where \( C \) is a positive constant.  This is the basic \( \G_2 \)-instanton on \( \bR^7 \).  \qed
\end{proposition}

The instantons on \( \G_2\times\bR^7 \) arising from Proposition \ref{prop:instanton-G2}
are well-known to physicists, arising in the context of
heterotic string theory (see, for example,
\cite{Nicolai-G:instanton,Ivanova-P:instanton}). In these papers, a more general ansatz is used, allowing for other gauge groups, but the only explicit solution
provided is the one given above.
More recently \cite{Driscoll:G2inst}, it has been shown that this basic instanton
is locally unique, but it is not known whether it is globally unique.

\begin{remark}
From the equations \eqref{eq:g2}, one sees that the \( \G_2 \)-instanton condition is conformally invariant, in the sense that the space of \( \G_2 \)-instantons are the same for all \( \G_2\)-structures defined by 3-forms which differ by multiplication by a positive function.  Therefore, it is natural to ask whether one can use stereographic projection to take the basic \( \G_2 \)-instanton and define a \( \G_2\)-instanton on \( \sS^7\) (perhaps away from the North pole) for its standard (or some other) \( \G_2\)-structure.  However, one finds that the \( G_2\)-structure one obtains away from the North pole on \( \sS^7 \) is not a multiple of the standard \( \G_2 \)-structure on \( \sS^7\), and in fact cannot be extended to the North pole as a \( \G_2 \)-structure.  Hence, the basic instanton on \( \bR^7 \) cannot be lifted in any obvious way to \( \sS^7 \), unlike what occurs for the analogous setting of \( \bR^4 \) and \( \sS^4 \).
\end{remark}

\subsection{On \( \bR^8 \)}
Finally, we come to the basic instanton on \( \bR^8 \). We may regard \( \Spin(7) \)
as the \( \GL(8,\bR) \)-stabiliser of the closed \( 4 \)-form
\begin{equation*}
  \Phi=\varphi\wedge dx^8+\Hodge_7\varphi,
\end{equation*}
where \( \varphi \)
is given by \eqref{eq:G2-3form} and \( \Hodge_7\varphi \)
by \eqref{eq:G2-4form}.

Again, we wish to write this as a conical structure on the open set
\( \bR^8\setminus\{0\}\cong \bR_+\times \sS^7 \):
\begin{equation}\label{eq:s7}
  \Phi=r^3\varphi\wedge dr+r^4\Hodge_7\varphi
\end{equation}
where \( \varphi \)
is a \( \G_2 \)-structure
on \( \sS^7 \)
satisfying \( d\varphi=-4\Hodge_7\varphi \),
in accordance with the torsion-free condition.

It is well-known that the condition for a \( 2 \)-form
\( \beta \)
to lie in \( \lie{spin}(7)\subset\Lambda^2(\bR^8)^* \) is given by:
\begin{equation}\label{eq:spin7}
  \Hodge(\Phi\wedge\beta)=-\beta.
\end{equation}	
Again, the link to the idea of anti-self-duality is clear from this condition.

For our connection ansatz \eqref{eq:conn-ansatz} we could again substitute \( \beta=\Omega \) in \eqref{eq:spin7} to obtain the \( \Spin(7) \)-instanton condition.  However, we may also rephrase the \( \Spin(7) \)-instanton
condition for the family \( A(r)\) of connections over \( \sS^7 \) in \eqref{eq:conn-ansatz} as 
\begin{equation}\label{eq:spin7-evol}
  A'=\Hodge_r(F_A\wedge r^4\Hodge_7\varphi)
\end{equation}
for details see \cite{Lotay-M:instantons}.  Here we used the decomposition \eqref{eq:s7} and \( \Hodge_r \) is the Hodge star on the 7-sphere of radius \( r\).

In order to obtain a basis of \( \lie{spin}(7)\subset\so(8)\), we proceed as follows. Using the elementary matrices \( E_{k\ell}\), a choice of orthogonal basis of
\( \lie{spin}(7)\subset\Lambda^2(\bR^8)^* \) is:
\begin{gather*}
    e_1=-E_{12}+E_{21}-E_{47}+E_{74},\quad e_2=\tfrac1{\sqrt3}(-E_{12}+E_{21}+E_{47}-E_{74}-2E_{56}+2E_{65}),\\
    e_3=\tfrac1{\sqrt6}(E_{12}-E_{21}-3E_{38}+3E_{83}-E_{47}+E_{74}-E_{56}+E_{56}),\displaybreak[0]\\
    e_4=-E_{13}+E_{31}-E_{46}+E_{64},\quad e_5=\tfrac1{\sqrt3}(E_{13}-E_{31}-E_{46}+E_{64}-2E_{57}+2E_{75}),\\
    e_6=\tfrac1{\sqrt6}(-E_{13}+E_{31}-3E_{28}+3E_{82}+E_{46}-E_{64}-E_{57}+E_{75}),\displaybreak[0]\\
    e_7=E_{14}-E_{41}-E_{27}+E_{72},\quad e_8=\tfrac1{\sqrt3}(E_{14}-E_{41}+E_{27}-E_{72}-2E_{58}+2E_{85}),\\
    e_9=\tfrac1{\sqrt6}(E_{14}-E_{41}+E_{27}-E_{72}-3E_{36}+3E_{63}+E_{58}-E_{85}),\displaybreak[0]\\
    e_{10}=-E_{15}+E_{51}-E_{48}+E_{84},\quad
    e_{11}=\tfrac1{\sqrt3}(E_{15}-E_{51}-2E_{26}+2E_{62}-E_{48}+E_{84}),\\
    e_{12}=\tfrac1{\sqrt6}(-E_{15}+E_{51}-E_{26}+E_{62}-3E_{37}+3E_{73}+E_{48}-E_{84}),\displaybreak[0]\\
    e_{13}=E_{16}-E_{61}-E_{78}+E_{87},\quad
    e_{14}=\tfrac1{\sqrt3}(-E_{16}+E_{61}-2E_{25}+2E_{52}-E_{78}+E_{88}),\\
    e_{15}=\tfrac1{\sqrt6}(-E_{16}+E_{61}+E_{25}-E_{52}-3E_{34}+3E_{43}-E_{78}+E_{87}),\displaybreak[0]\\
    e_{16}=-E_{17}+E_{71}-E_{68}+E_{86},\quad
    e_{17}=\tfrac1{\sqrt3}(-E_{17}+E_{71}-2E_{24}+2E_{42}+E_{68}-E_{86}),\\
    e_{18}=\tfrac1{\sqrt6}(E_{17}-E_{71}-E_{24}+E_{42}-3E_{35}+3E_{53}-E_{68}+E_{86}),\displaybreak[0]\\
    e_{19}=E_{18}-E_{81}-E_{23}+E_{32},\quad
    e_{20}=\tfrac1{\sqrt3}(E_{18}-E_{81}-2E_{45}+2E_{54}+E_{23}-E_{32}),\\
    e_{21}=\tfrac1{\sqrt6}(E_{18}-E_{81}+E_{23}-E_{32}+E_{45}-E_{54}-3E_{67}+3E_{76}).
\end{gather*}
The corresponding basis of \( 2 \)-forms is:
\begin{gather*}
    \beta_1=dx^{12}+dx^{47},\quad \beta_2=\tfrac1{\sqrt3}(dx^{12}-dx^{47}+2dx^{56}),\\
    \beta_3=\tfrac1{\sqrt6}(-dx^{12}+3dx^{38}+dx^{47}+dx^{56}),\displaybreak[0]\\
    \beta_4=dx^{13}+dx^{46},\quad\beta_5=\tfrac1{\sqrt3}(-dx^{13}+dx^{46}+2dx^{57}),\\
    \beta_6=\tfrac1{\sqrt6}(dx^{13}+3dx^{28}-dx^{46}+dx^{57})
\displaybreak[0]\\
    \beta_7=-dx^{14}+dx^{27},\quad\beta_8=-\tfrac1{\sqrt3}(dx^{14}+dx^{27}-2dx^{58}),\\
    \beta_9=\tfrac1{\sqrt6}(-dx^{14}-dx^{27}+3dx^{36}-dx^{58}),
\displaybreak[0]\\
    \beta_{10}=dx^{15}+dx^{48},\quad
    \beta_{11}=\tfrac1{\sqrt3}(-dx^{15}+2dx^{26}+dx^{48}),\\
    \beta_{12}=\tfrac1{\sqrt6}(dx^{15}+dx^{26}+3dx^{37}-dx^{48}),
\displaybreak[0]\\
    \beta_{13}=-dx^{16}+dx^{78},\quad
    \beta_{14}=\tfrac1{\sqrt3}(dx^{16}+2dx^{25}+dx^{78}),\\
    \beta_{15}=\tfrac1{\sqrt6}(dx^{16}-dx^{25}+3dx^{34}+dx^{78}),
\displaybreak[0]\\
    \beta_{16}=dx^{17}+dx^{68},\quad
    \beta_{17}=\tfrac1{\sqrt3}(dx^{17}+2dx^{24}-dx^{68}),\\
    \beta_{18}=\tfrac1{\sqrt6}(-dx^{17}+dx^{24}+3dx^{35}+dx^{68}),
\displaybreak[0]\\
    \beta_{19}=-dx^{18}+dx^{23},\quad
    \beta_{20}=\tfrac1{\sqrt3}(-dx^{18}+2dx^{45}-dx^{23}),\\
    \beta_{21}=-\tfrac1{\sqrt6}(dx^{18}+dx^{23}+dx^{45}-3dx^{67}).
\end{gather*}

To illustrate computations, let us consider the point \( x\in\mathcal{S}^7 \) with \( x_8=1 \) and \( x_j=0 \) for all other \(j\).
 We then have that, at \( x \), 
\begin{gather*}
    \alpha_1=0=\alpha_2,\quad\alpha_3=-\tfrac{\sqrt3}{\sqrt2}dx^3,\quad \alpha_4=0=\alpha_5,\displaybreak[0]\\
    \alpha_6=-\tfrac{\sqrt3}{\sqrt2}dx^2,\quad\alpha_7=0,\quad\alpha_8=-\tfrac2{\sqrt3}dx^5,\quad
    \alpha_9=\tfrac1{\sqrt6}dx^5,\displaybreak[0]\\
    \alpha_{10}=-dx^4,\quad\alpha_{11}=-\tfrac1{\sqrt3}dx^4,\quad
    \alpha_{12}=\tfrac1{\sqrt6}dx^4,\displaybreak[0]\\
    \alpha_{13}=-dx^7,\quad\alpha_{14}=-\tfrac1{\sqrt3}dx^7,\quad\alpha_{15}=-\tfrac1{\sqrt6}dx^7,\quad
    \alpha_{16}=-dx^6,\displaybreak[0]\\ \alpha_{17}=\tfrac1{\sqrt3}dx^6,\quad
    \alpha_{18}=-\tfrac1{\sqrt6}dx^6,\quad
    \alpha_{19}=dx^1,\quad\alpha_{20}=\tfrac1{\sqrt3}dx^1,\quad\alpha_{21}=\tfrac1{\sqrt6}dx^1.
  \end{gather*}
In addition, we have
\begin{gather*}
    \beta_1=dx^{12}+dx^{47},\quad \beta_2=\tfrac1{\sqrt3}(dx^{12}-dx^{47}+2dx^{56}),\quad
    \beta_3=\tfrac1{\sqrt6}(-dx^{12}+dx^{47}+dx^{56}),
\displaybreak[0]\\
    \beta_4=dx^{13}+dx^{46},\quad\beta_5=\tfrac1{\sqrt3}(-dx^{13}+dx^{46}+2dx^{57}),\quad
    \beta_6=\tfrac1{\sqrt6}(dx^{13}-dx^{46}+dx^{57}),
\displaybreak[0]\\
    \beta_7=-dx^{14}+dx^{27},\quad\beta_8=-\tfrac1{\sqrt3}(dx^{14}+dx^{27}),\quad
    \beta_9=\tfrac1{\sqrt6}(-dx^{14}-dx^{27}+3dx^{36}),
\displaybreak[0]\\
    \beta_{10}=dx^{15},\quad \beta_{11}=\tfrac1{\sqrt3}(-dx^{15}+2dx^{26}),\quad
    \beta_{12}=\tfrac1{\sqrt6}(dx^{15}+dx^{26}+3dx^{37}),
\displaybreak[0]\\
    \beta_{13}=-dx^{16},\quad \beta_{14}=\tfrac1{\sqrt3}(dx^{16}+2dx^{25}),\quad
    \beta_{15}=\tfrac1{\sqrt6}(dx^{16}-dx^{25}+3dx^{34}),
\displaybreak[0]\\
    \beta_{16}=dx^{17},\quad \beta_{17}=\tfrac1{\sqrt3}(dx^{17}+2dx^{24}),\quad
    \beta_{18}=\tfrac1{\sqrt6}(-dx^{17}+dx^{24}+3dx^{35}),
\displaybreak[0]\\
    \beta_{19}=dx^{23},\quad \beta_{20}=\tfrac1{\sqrt3}(2dx^{45}-dx^{23}),\quad
    \beta_{21}=-\tfrac1{\sqrt6}(dx^{23}+dx^{45}-3dx^{67}).
  \end{gather*}

We now consider the evolution equation \eqref{eq:spin7-evol} for our connection ansatz \eqref{eq:conn-ansatz}. Focusing on the coefficient of \( e_3 \),
we have that the projection of \( A'\), \( \Hodge_7(dA\wedge r^4\Hodge_7\varphi) \), \( \Hodge_7(\tfrac12[A,A]\wedge r^4\Hodge_7\varphi) \)  are given by
\begin{equation*}
  -\tfrac{\sqrt3}{\sqrt2}\frac{d(ra)}{d r}dx^3, \quad   -2\tfrac{\sqrt3}{\sqrt2}a(r)dx^3,\quad   -3\tfrac{\sqrt3}{\sqrt2}ra^2dx^3,
\end{equation*}
respectively.
These considerations suggest that the \(\Spin(7)\)-instanton evolution equations \eqref{eq:spin7-evol} for the ansatz \eqref{eq:conn-ansatz} is:
\begin{equation*}
  \frac{d a}{d r}=\frac{a}{r}(1+3ar).
\end{equation*}
Indeed, it can be checked that the equations obtained from the remaining \( 20 \) components agree with this ODE. 

Summing up, we have shown:
\begin{proposition}
  \label{prop:instanton-Spin7}
 The ansatz \eqref{eq:conn-ansatz} on \( \bR^8 \) with \(\G=\Spin(7)\) yields a non-trivial \( \Spin(7) \)-instanton
  provided
  \begin{equation*}
    a(r)=-\frac23\frac{r}{r^2+C},
  \end{equation*}
  where \( C \) is a positive constant. This is the basic \( \Spin(7)\)-instanton on \( \bR^8\). \qed
\end{proposition}

As for the \( \G_2\)-case, the corresponding family of instantons on \( \Spin(7)\times\bR^8 \)
has been known to physicists for decades (see, for example,
\cite{Fairlie-N:instanton,Fubini-al:Spin7}).  Moreover, in this case a more general ansatz is used in the aforementioned papers, but only one explicit solution is provided, agreeing with the above.  No uniqueness result is yet known for the basic instanton on \( \bR^8 \): even local uniqueness is currently open.

\begin{remark}
Just as for the \( \G_2 \) case, the \( \Spin(7)\)-instanton condition is conformally invariant, in that it does not change when multiplying the defining \( 4\)-form for the \( \Spin(7)\)-structure by a positive function.  However, here the possibility of lifting the basic instanton on \( \bR^8\) to \( \sS^8\) does not even arise  as there is no \( \Spin(7)\)-structure on \( \sS^8\). This follows since the \( 8 \)-sphere has non-vanishing Euler characteristic, \(\mathcal X \), but vanishing pontryagin classes, \( p_i \), and so cannot satisfy the condition
\begin{equation*}
8\mathcal{X}=4p_2-p_1^2,
\end{equation*}
which is the necessary and sufficient condition for an orientable spin \( 8 \)-manifold to admit a \( \Spin(7) \)-structure (cf. \cite{GrayGreen:Spin7}).
\end{remark}

\section{Nahm's equations and quaternionic structures}
Prompted by \cref{rem:basicansatz}, we will now turn to an alternative ansatz
that has appeared in the physics literature (see \cite{Ivanova-P:instanton} and the references therein). This will allow us to construct families of explicit instantons on \( \bH^n=\bR^{4n} \), essentially rephrasing \cite[Sec. 5.1, Proposition 1]{Ivanova-P:instanton} in index free notation. 
In contrast with the basic instanton constructions, the gauge group \( \LK \) is not imposed to be the same as the internal symmetry group \( \G \). 

In this section, we let \( (J^1,J^2,J^3) \) be the standard triple of complex structures on \( \bH^n=\bR^{4n} \) and let \(\{ k_j \} \) be any basis of the Lie algebra \( \lk \) of \( \LK \).

\subsection{Nahm's equations and instantons on \( \bH \)}
Consider first the task of finding an instanton on \(\mathcal U\subset \bR^4 \) with given gauge group \( \LK \).   For a scalar function \( f \) defined on \( \mathcal U \), we make the following connection ansatz
\begin{equation}
\label{eq:ansatz4D}
\omega=\sum_{j=1}^3T_j(f)J^j(df),
\end{equation}
where 
\begin{equation*} 
T_i(s)=\sum_{j=1}^{\dim \lk}T_i^j(s)k_j,\qquad i=1,2,3, 
\end{equation*}
is a \( \lk \)-valued function defined on an interval \( I\subset\mathbb R \), containing the range of \( f \). 
For this ansatz, we have the following recipe for consructing self-dual and instantons, sligtly improving the statement of \cite[Sec 6.1, Prop. 2]{Ivanova-P:instanton}:
\begin{proposition}\label{prop:Nahm}
Consider a connection \( \omega \) on a principal \( \LK \)-bundle over \( \bR^4 \) given by the ansatz \eqref{eq:ansatz4D}.
If \( T=(T_1,T_2,T_3) \), \( T_i\colon I\to\lk \), satisfies Nahm's equations, 
\begin{equation}
\label{eq:Nahms1}
T'_i(s)=[T_j(s),T_k(s)],
\end{equation}
for each cyclic permutation \( (ijk)=(123) \). Then:
\begin{enumerate}
\item \( \omega \) is self-dual if and only if \( f \) satisfies the system: 
\begin{equation}
\label{eq:SD}
\frac{\partial^2 f}{\partial x_{pp}^2}=\frac{\partial^2 f}{\partial x_{qq}^2}\qquad\textrm{and}\qquad
\frac{\partial^2 f}{\partial x_p\partial x_q}=0,
\end{equation}
for all \( 1\leqslant p<q\leqslant4 \).
\item \( -\omega \) is anti-self-dual if and only if \( f \) is harmonic.
\end{enumerate}
\end{proposition}

\begin{proof}
The curvature of \( \omega \) is given by
\begin{equation}
\label{eq:curvAnsatz}
\Omega=\sum_{i=1}^3\left(T'_i(f)df\wedge J^i(df)+T_i(f)d\left(J^i(df)\right)+[T_j(f),T_k(f)]J^j(df) \wedge J^k(df)\right),
\end{equation}
where, as above, \((ijk)=(123) \) denotes a cyclic permutation of \( (123) \).
To proceed, we note that more explicitly,
\begin{equation*}
\begin{gathered}
J^1(df)=-f_2dx_1+f_1dx_2-f_4dx_3+f_3dx_4,\\
J^2(df)=-f_3dx_1+f_4dx_2+f_1dx_3-f_2dx_4,\\
J^3(df)=-f_4dx_1-f_3dx_2+f_2dx_3+f_1dx_4.
\end{gathered}
\end{equation*}
Here we have used the notation 
\( df=\sum_{j=1}^4f_jdx_j \), where \( f_j=\partial f/\partial x_j\). In particular, we have that
\begin{equation*}
\Hodge df\wedge J^i(df)=J^j(df)\wedge J^k(df),
\end{equation*}
with \((ijk)=(123) \). Therefore, if \( T \) satisfies Nahm's equations, then self-duality becomes equivalent to the condition
\begin{equation*}
d(J^i(df))=\Hodge d(J^i(df)),\qquad i=1,2,3,
\end{equation*}
which, more explicitly, can be phrased in terms of the system \eqref{eq:SD}. Similarly, anti-selfduality for \( -\omega \) reads,
\begin{equation*}
d(J^i(df))=-\Hodge d(J^i(df)),\qquad i=1,2,3,
\end{equation*}
which is equivalent to requiring that \( f \) is harmonic.
\end{proof}

\begin{remark}
A well-known solution, \( T_i\colon(0,\infty)\to\lk \), to Nahm's equations, sometimes referred to as the Nahm pole solution,  takes the form:
\begin{equation*}
T_i(s)=\frac1{2s}k_i,\qquad i=1,2,3,
\end{equation*}
where \( k_1,k_2,k_3 \) span an \( \su(2) \)-subalgebra of \( \lk \) and satisfy the commutation relations
\( [k_1,k_2]=-2k_3 \), etc. Recently, this solution has played a role in the study of the Kapustin-Witten equations (see, e.g., \cite{Mazzeo-Witten:KWeq,He-Mazzeo:KWeq}).
\end{remark}

Proposition \ref{prop:Nahm} allows us to produce families of explicit instantons on \( \bR^4 \). Some basic examples are given below.

\begin{example}
Given a solution of 
Nahm's equations and real constants \( A \), \( B_j \), \( C \), any function of the form
\begin{equation*}
f(\mathbf x)=A\mathbf x\cdot \mathbf x+\sum_{j=1}^4B_jx_j+C
\end{equation*}
produces an SD instanton on \( \bR^4 \). 
\end{example}

\begin{example}\label{eq:ASDdim4}
For the ASD case, we can produce instantons on \( \bR^4 \) starting from a solution to Nahm's equation and any function of the form
\begin{equation*}
\begin{gathered}
f(\mathbf x)=A(\mathbf x\cdot\mathbf x)^{-1}+A_1x_1^2+A_2x_2^2+A_3x_3^2-(A_1+A_2+A_3)x_4^2\\
+\sum_{1\leqslant i<j\leqslant 4}B_{ij}x_ix_j+\sum_{j=1}^4B_jx_j+C,
\end{gathered}
\end{equation*}
with \( A \), \( A_i \), \( B_{ij} \), \( B_i \) and \( C \) fixed real numbers; obviously, \( f \) is only defined away from zero if \( A\neq0 \).
\end{example}

\subsection{Quaternionic instantons on \(\mathbb H^n \)}
The connection ansatz \eqref{eq:ansatz4D} can be used, more generally, for \( \mathbb H^n=\bR^{4n} \). The associated curvature \( 2 \)-form is still given by \eqref{eq:curvAnsatz},
where, generalising the notation used above, we have: 
\begin{equation*}
\begin{gathered}
J^1(df)=\sum_{j=0}^{n-1}-f_{4j+2}dx_{4j+1}+f_{4j+1}dx_{4j+2}-f_{4j+4}dx_{4j+3}+f_{4j+3}dx_{4j+4},\\
J^2(df)=\sum_{j=0}^{n-1}-f_{4j+3}dx_{4j+1}+f_{4j+4}dx_{4j+2}+f_{4j+1}dx_{4j+3}-f_{4j+2}dx_{4j+4},\\
J^3(df)=\sum_{j=0}^{n-1}-f_{4j+4}dx_{4j+1}-f_{4j+3}dx_{4j+2}+f_{4j+2}dx_{4j+3}+f_{4j+1}dx_{4j+4}.
\end{gathered}
\end{equation*}

In order to generalise the notion of ASD instantons (see, for example, \cite{Carrion:Instanton}),
we introduce the "fundamental" \( 4 \)-form, \( \Sigma \),  on \( \bR^{4n} \) given by
\begin{equation*}
\Sigma=\tfrac12(\sigma_1^2+\sigma_2^2+\sigma_3^2),
\end{equation*}
where
\begin{equation*}
\begin{gathered}
\sigma_1=\sum_{j=0}^{n-1}dx_{4j+1,4j+2}+dx_{4j+3,4j+4},\quad
\sigma_2=\sum_{j=0}^{n-1}dx_{4j+1,4j+3}+dx_{4j+4,4j+2},\\
\sigma_3=\sum_{j=0}^{n-1}dx_{4j+1,4j+4}+dx_{4j+2,4j+3}.
\end{gathered}
\end{equation*}
In these terms, we have the following definition.

\begin{definition}\label{def:ASD} A generalisation of ASD instantons on \( \bH^n\) are given by connections
\( \omega \) whose curvature \(\Omega \) satisfies:
\begin{equation}\label{eq:quat2}
\Omega\wedge\Sigma^{n-1}=c_1\Hodge \Omega,
\end{equation}
where \( c_1=-(2n-1)!/2^{n-1} \). In this case, the curvature \( 2 \)-form
takes values in \( \lsp(n)\subset\lsp(1)\oplus\lsp(n) \).
\end{definition}

\cref{prop:Nahm} generalises to higher dimensional flat quaternionic space, \( \bH^n \), as follows.

\begin{proposition}\label{prop:NahmHn}
Consider a connection on a principal \( \LK \)-bundle over \( \bH^n \) given by the negative, \( -\omega \), of the ansatz \eqref{eq:ansatz4D}.
If \( T=(T_1,T_2,T_3) \), \( T_i\colon I\to\lk \), satisfies Nahm's equations \eqref{eq:Nahms1}, then \( -\omega \) is anti-self-dual if and only if \( f \) satisfies
the second order equations:
\begin{equation}\label{eq:ASDHnSecOrd}
d(J^i(df))\wedge\Sigma^{n-1}=c_1\Hodge d(J^i(df)),\qquad i=1,2,3,
\end{equation}
where \( c_1=-(2n-1)!/2^{n-1} \).
\qed
\end{proposition}

As in the \( 4 \)-dimensional case, it is easy to construct families of explicit solutions as illustrated below for \( \bH^2 \).

\begin{example}
In higher dimensions, the conditions \eqref{eq:ASDHnSecOrd} are more involved than for \( \bR^4 \). For example, in the case of \( \bH^2 \) this system
of second order PDEs for \( f \) can be phrased as:
\begin{equation*}
\begin{gathered}
\frac{\partial^2f}{\partial x_1^2}+\frac{\partial^2f}{\partial x_2^2}+\frac{\partial^2f}{\partial x_3^2}+\frac{\partial^2f}{\partial x_4^2}=0=\frac{\partial^2f}{\partial x_5^2}+\frac{\partial^2f}{\partial x_6^2}+\frac{\partial^2f}{\partial x_7^2}+\frac{\partial^2f}{\partial x_8^2}\\
\frac{\partial^2f}{\partial x_1\partial x_5}+\frac{\partial^2f}{\partial x_2\partial x_6}+\frac{\partial^2f}{\partial x_3\partial x_7}+\frac{\partial^2f}{\partial x_4\partial x_8}=0=\frac{\partial^2f}{\partial x_1\partial x_6}-\frac{\partial^2f}{\partial x_2\partial x_5}-\frac{\partial^2f}{\partial x_3\partial x_8}+\frac{\partial^2f}{\partial x_4\partial x_7}\\
\frac{\partial^2f}{\partial x_1\partial x_7}+\frac{\partial^2f}{\partial x_2\partial x_8}-\frac{\partial^2f}{\partial x_3\partial x_5}-\frac{\partial^2f}{\partial x_4\partial x_6}=0=\frac{\partial^2f}{\partial x_1\partial x_8}-\frac{\partial^2f}{\partial x_2\partial x_7}+\frac{\partial^2f}{\partial x_3\partial x_6}-\frac{\partial^2f}{\partial x_4\partial x_5}.
\end{gathered}
\end{equation*}
From these equations, it follows that, with any solution of Namh's equation, any function of the form
\begin{equation*}
\begin{gathered}
f(\mathbf x)=f_1(\mathbf x)+f_2(\mathbf x)+C_1x_1x_5+C_2x_2x_6+C_3x_3x_7-(C_1+C_2+C_3)x_4x_8\\
+D_1x_1x_6+D_2x_2x_5+D_3x_3x_8-(D_1-D_2-D_3)x_4x_7\\
+E_1x_1x_7+E_2x_2x_8+E_3x_3x_5+(E_1+E_2-E_3)x_4x_6\\
+F_1x_1x_8+F_2x_2x_7+F_3x_3x_6+(F_1-F_2+F_3)x_4x_5\\
\end{gathered}
\end{equation*}
where \( f_1 \) and \( f_2 \) are solutions of the type of \cref{eq:ASDdim4} restricted to the \( J_i \) invariant subspaces \( \{x_1,x_2,x_3,x_4\} \) and \( \{ x_5,x_6,x_7,x_8\}\), respectively, and \( C_i \), \( D_i \), \( E_i \) and \( F_i \) are constants. 
\end{example}

\begin{remark}
In view of \cref{prop:NahmHn}, it seems natural to ask for a generalisation of SD instantons on \( \bH^n\), using the ansatz \eqref{eq:curvAnsatz}. This would be
a connection \( \omega \) whose curvature \(\Omega \) satisfies:
\begin{equation}\label{eq:quat1}
\Omega\wedge\Sigma^{n-1}=c_2\Hodge \Omega,
\end{equation}
where \( c_2=(2n+1)!/(6n 2^{n-1}) \). This corresponds to the curvature \( 2 \)-form
taking values in \( \lsp(1)\subset\lsp(1)\oplus\lsp(n) \).

Direct computations show that the part of \cref{prop:Nahm} concerning SD instantons does not generalise to higher dimensional \( \bH^n \).
\end{remark}

In a sense, it is slightly misleading referring to the above instantons as being "quaternionic", since we are dealing with a flat hyperk\"ahler structure on \( \bH^n \) rather than
a genuine quaternionic structure. Explicit examples of quaternionic ASD instantons are known, e.g., on \( \bH P^n \), cf.~\cite{Mamone-Salamon:qKinstanton}; note that the authors of this paper use the opposite terminology for quaternionic ASD and SD instantons, respectively. These latter examples were revisited more recently in \cite{Devchand-al:qKinstanton} in the context of hyperk\"ahler cones. 

\newpage


\providecommand{\bysame}{\leavevmode\hbox to3em{\hrulefill}\thinspace}
\providecommand{\MR}{\relax\ifhmode\unskip\space\fi MR }
\providecommand{\MRhref}[2]{%
  \href{http://www.ams.org/mathscinet-getitem?mr=#1}{#2}
}
\providecommand{\href}[2]{#2}

\end{document}